\documentclass{amsart}

\usepackage[top=40truemm,bottom=40truemm,left=35truemm,right=35truemm]{geometry}

\usepackage{amsmath}
\usepackage{amssymb}
\usepackage{amsthm}
\usepackage{ascmac}
\usepackage{cases}
\usepackage{multicol}
\usepackage[all]{xy}
\usepackage{nccmath}

\usepackage{graphicx}

\usepackage{amscd} 
\usepackage{float}
\usepackage{bm} 
\usepackage{type1cm} 


\newtheorem{theorem}{Theorem}[section]
\newtheorem{proposition}[theorem]{Proposition}
\newtheorem{lemma}[theorem]{Lemma}
\newtheorem{corollary}[theorem]{Corollary}
\theoremstyle{definition}
\newtheorem{remark}[theorem]{Remark}
\newtheorem{example}[theorem]{Example}

\newtheorem{definition}[theorem]{Definition}

\newcommand{\os}{\mathbin{\overline{*}}}
\newcommand{\us}{\mathbin{\underline{*}}}
\newcommand{\type}{\operatorname{type}}
\newcommand{\flow}{\mathrm{Flow}}
\newcommand{\col}{\mathrm{Col}}
\newcommand{\rank}{\operatorname{rank}}
\newcommand{\ol}{\overline}

\begin{document}

\title{The Gordian distance of handlebody-knots and Alexander biquandle colorings}

\author{Tomo \textsc{Murao}}
\address[T. Murao]
{Institute of Mathematics, University of Tsukuba,\\
1-1-1 Tennoudai, Tsukuba, Ibaraki 305-8571, Japan.}
\email{t-murao@math.tsukuba.ac.jp}

\subjclass[2010]{Primary 57M25; Secondary 57M15, 57M27}

\keywords{handlebody-knot, biquandle, Gordian distance, unknotting number}

\date{}

\begin{abstract}
We give lower bounds for the Gordian distance and the unknotting number 
of handlebody-knots 
by using Alexander biquandle colorings.
We construct handlebody-knots with Gordian distance $n$ 
and unknotting number $n$ for any positive integer $n$.
\end{abstract}

\maketitle

\section{Introduction}
The Gordian distance of two classical knots 
is the minimal number of crossing changes neened to be deformed each other.
In particular, 
we call the Gordian distance of a classical knot and the trivial one 
the unknotting number of the classical knot.
Clark, Elhamdadi, Saito and Yeatman \cite{CESY14} gave a lower bound for 
the Nakanishi index \cite{Nakanishi81}, 
which induced a lower bound for 
the unknotting number of classical knots.
This is an generalization of the Przytycki's result \cite{Przytycki98}.
In this paper, 
we give lower bounds for the Gordian distance and the unknotting number 
of handlebody-knots, 
which is a generalization of a classical knot with respect to a genus.

Ishii \cite{Ishii08} introduced an enhanced constituent link of a spatial trivalent graph, 
and Ishii and Iwakiri \cite{II12} introduced an $A$-flow of a spatial graph, 
where $A$ is an abelian group, 
to define colorings and invariants of handlebody-knots.
Iwakiri \cite{Iwakiri15} gave a lower bound for the unknotting number of handlebody-knots 
by using Alexander quandle colorings of 
its $\mathbb{Z}_2$ or $\mathbb{Z}_3$-flowed diagram.
Ishii, Iwakiri, Jang and Oshiro \cite{IIJO13}
introduced a $G$-family of quandles, 
which is an extension of the above structures.
Recently, Ishii and Nelson \cite{IN17} introduced a $G$-family of biquandles, 
which is a biquandle version of a $G$-family of quandles.

In this paper, 
we extend the result in \cite{Iwakiri15} in three directions.
First, we extend from $\mathbb{Z}_2$, $\mathbb{Z}_3$-flows 
to any $\mathbb{Z}_m$-flow.
Second, we extend from quandles to biquandles.
Finally, we extend from unknotting numbers to Gordian distances.
Thus we can determine the Gordian distance and the unknotting number 
of handlebody-knots more efficiently.
We construct handlebody-knots with Gordian distance $n$ 
and unknotting number $n$ for any $n \in \mathbb{Z}_{>0}$ 
and note that 
one of them can not be obtained by using Alexander quandle colorings 
introduced in \cite{Iwakiri15}.

This paper is organized into seven sections.
In Section 2, 
we recall the definition of a handlebody-knot 
and introduce the Gordian distance and the unknotting number of 
handlebody-knots.
In Section 3, 
we recall the definition of a (bi)quandle and 
a $G$-family of (bi)quandles.
In Section 4,
we introduce a coloring of a diagram of a handlebody-knot 
by using a $G$-family of biquandles.
In Section 5, 
we show that 
there are linear relationships 
for Alexander biquandle colorings of 
a diagram of a handlebody-knot.
In Section 6,
we give lower bounds for the Gordian distance and the unknotting number 
of handlebody-knots 
by using $\mathbb{Z}_m$-family of Alexander biquandles colorings.
In section 7,
we construct handlebody-knots with Gordian distance $n$ 
and unknotting number $n$ for any $n \in \mathbb{Z}_{>0}$.
Moreover, we note that 
one of them can not be obtained by using Alexander quandle colorings 
with $\mathbb{Z}_2, \mathbb{Z}_3$-flows introduced in \cite{Iwakiri15}.

\section{The Gordian distance of handlebody-knots}
A \emph{handlebody-link}, 
which is introduced in \cite{Ishii08}, 
is the disjoint union of handlebodies embedded in the 3-sphere $S^3$.
A \emph{handlebody-knot} is a handlebody-link with one component.
In this paper, 
we assume that every component of a handlebody-link 
is of genus at least $1$.
An \emph{$S^1$-orientation} of a handlebody-link is an orientation 
of all genus 1 components of the handlebody-link, 
where an orientation of a solid torus is an orientation of its core $S^1$.
Two $S^1$-oriented handlebody-links are \emph{equivalent} 
if there exists an orientation-preserving self-homeomorphism of $S^3$ 
sending one to the other 
preserving the $S^1$-orientation.

A \emph{spatial trivalent graph} is 
a graph whose vertices are valency 3 
embedded in $S^3$.
In this paper, 
a trivalent graph may have a circle component, 
which has no vertices.
A \emph{Y-orientation} of a spatial trivalent graph is 
a direction of all edges of the graph 
satisfying that every vertex of the graph is both 
the initial vertex of a directed edge 
and the terminal vertex of a directed edge (Figure \ref{Y-orientation}).
A vertex of a Y-oriented spatial trivalent graph can be allocated a sign; 
the vertex is said to be positive or negative, 
or to have sign $+1$ or $-1$. 
The standard convention is shown in Figure \ref{Y-orientation}.
For a Y-oriented spatial trivalent graph $K$ 
and an $S^1$-oriented handlebody-link $H$, 
we say that 
$K$ \emph{represents} $H$ 
if $H$ is a regular neighborhood of $K$ 
and the $S^1$-orientation of $H$ agrees with the Y-orientation.
Then any $S^1$-oriented handlebody-link can be represented by 
some Y-oriented spatial trivalent graph.
We define a \emph{diagram} of an $S^1$-oriented handlebody-link 
by a diagram of a Y-oriented spatial trivalent graph 
representing the handlebody-link.
An $S^1$-oriented handlebody-link is \emph{trivial}
if it has a diagram with no crossings.
Then the following theorem holds.

\begin{figure}[htb]
\begin{center}
\includegraphics[width=50mm]{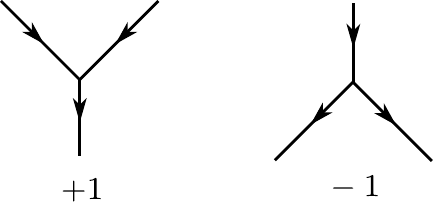}
\end{center}
\caption{Y-orientations and signs.}\label{Y-orientation}
\end{figure}

\begin{theorem}[\cite{Ishii15-2}]\label{Reidemeister moves}
For a diagram $D_i$ of an $S^1$-oriented handlebody-link $H_i$ $(i=1,2)$, 
$H_1$ and $H_2$ are equivalent 
if and only if 
$D_1$ and $D_2$ are related 
by a finite sequence of R1--R6 moves 
depicted in Figure \ref{Reidemeister move} 
preserving Y-orientations.
\end{theorem}

\begin{figure}[htb]
\begin{center}
\includegraphics[width=125mm]{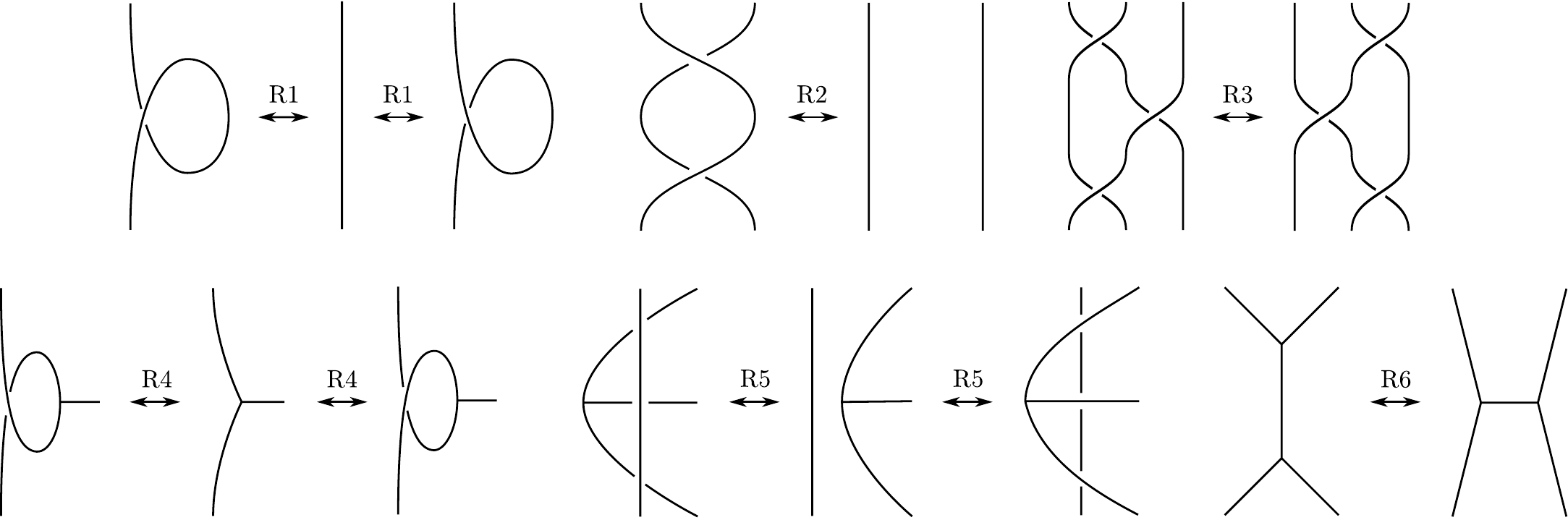}
\caption{The Reidemeister moves for handlebody-links.}\label{Reidemeister move}
\end{center}
\end{figure}

In this paper, 
for a diagram $D$ of an $S^1$-oriented handlebody-link, 
we denote by $\mathcal{A}(D)$ and $\mathcal{SA}(D)$ 
the set of all arcs of $D$ and the one of all semi-arcs of $D$ respectively, 
where a semi-arc is a piece of a curve 
each of whose endpoints is a crossing or a vertex.
An orientation of a (semi-)arc of $D$ 
is also represented by the normal orientation 
obtained by rotating the usual orientation counterclockwise 
by $\pi/2$ on the diagram.
For any $m \in \mathbb{Z}_{\geq 0}$,
we put 
$\mathbb{Z}_m:=\mathbb{Z}/m\mathbb{Z}$.

A \emph{crossing change} of an $S^1$-oriented handlebody-link $H$ is 
that of a spatial trivalent graph representing $H$.
This deformation can be realized by switching two handles 
depicted in Figure \ref{crossing change}.
It is easy to see that 
any two $S^1$-oriented handlebody-knots of the same genus can be related 
by a finite sequence of crossing changes.
For any two $S^1$-oriented handlebody-knots $H_1$ and $H_2$ of the same genus, 
we define their \emph{Gordian distance} $d(H_1,H_2)$ 
by the minimal number of crossing changes 
neened to be deformed each other.
In particular, 
for any $S^1$-oriented handlebody-knot $H$ 
and the $S^1$-oriented trivial handlebody-knot $O$ of the same genus, 
we define $u(H):=d(H,O)$, 
which is called the \emph{unknotting number} of $H$.

\begin{figure}[htb]
\begin{center}
\includegraphics[width=55mm]{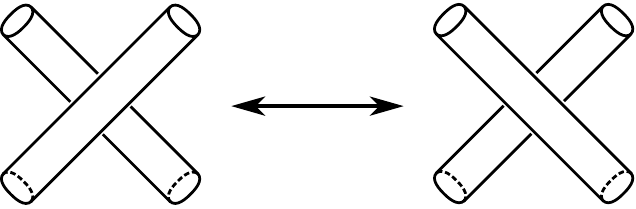}
\caption{A crossing change of an $S^1$-oriented handlebody-link.}\label{crossing change}
\end{center}
\end{figure}

\section{A biquandle and a $G$-family of biquandles}
We recall the definitions of a quandle and a biquandle.

\begin{definition}[\cite{Joyce82, Matveev82}]
A quandle is a non-empty set $X$ with a binary operation $* : X \times X \to X$ 
satisfying the following axioms.
\begin{itemize}
\item
For any $x \in X$, 
$x*x=x$.
\item
For any $x \in X$, 
the map $S_x : X \to X$ defined by $S_x(y)=y*x$ is a bijection.
\item
For any $x,y,z \in X$, 
$(x*y)*z=(x*z)*(y*z)$.
\end{itemize}
\end{definition}

\begin{definition}[\cite{FRS95}]
A biquandle is a non-empty set $X$ 
with binary operations $\os, \us : X \times X \to X$ 
satisfying the following axioms.
\begin{itemize}
\item
For any $x \in X$, 
$x \us x=x \os x$.
\item
For any $x \in X$, 
the map $\underline{S}_x : X \to X$ 
defined by $\underline{S}_x(y)=y \us x$ is a bijection.
\item[]
For any $x \in X$, 
the map $\overline{S}_x : X \to X$ 
defined by $\overline{S}_x(y)=y \os x$ is a bijection.
\item[]
The map $S : X \times X \to X \times X$ 
defined by $S(x,y)=(y \os x,x \us y)$ is a bijection.
\item
For any $x,y,z \in X$, 
\begin{align*}
(x \us y) \us (z \us y)=(x \us z) \us (y \os z),\\
(x \us y) \os (z \us y)=(x \os z) \us (y \os z),\\
(x \os y) \os (z \os y)=(x \os z) \os (y \us z).
\end{align*}
\end{itemize}
\end{definition}

We define $\us^nx:=\underline{S}_x^n$ and $\os^nx:=\overline{S}_x^n$ 
for any $n \in \mathbb{Z}$.
We note that 
$(X,*)$ is a quandle 
if and only if 
$(X,*,\os)$ is a biquandle 
with $x \os y=x$.
For any $m \in \mathbb{Z}_{\geq 0}$, 
a $\mathbb{Z}_m[s^{\pm1},t^{\pm1}]$-module $X$ is a biquandle
with $a \us b=ta+(s-t)b$ and $a \os b=sa$, 
which we call an \emph{Alexander biquandle}.
When $s=1$, 
an Alexander biquandle  
coincides with an Alexander quandle.

\begin{definition}[\cite{IIKKMOpre}]
Let $X$ be a biquandle.
We define two families of binary operations 
$\us^{[n]}, \os^{[n]} : X \times X \to X (n \in \mathbb{Z})$ 
by the equalities 
\begin{align*}
& a \us^{[0]}b=a,~~ a \us^{[1]}b=a \us b,~~ a \us^{[i+j]}b=(a\us^{[i]}b)\us^{[j]}(b\us^{[i]}b),\\
& a \os^{[0]}b=a,~~ a \os^{[1]}b=a \os b,~~ a \os^{[i+j]}b=(a\os^{[i]}b)\os^{[j]}(b\os^{[i]}b)
\end{align*}
for any $i,j \in \mathbb{Z}$.
\end{definition}

Since $a=a \us^{[0]}b=(a\us^{[-1]}b)\us^{[1]}(b\us^{[-1]}b)=(a\us^{[-1]}b)\us(b\us^{[-1]}b)$, 
we have $a \us^{[-1]}b=a\us^{-1}(b\us^{[-1]}b)$ 
and $(b\us^{[-1]}b)\us(b\us^{[-1]}b)=b$.
Then for an Alexander biquandle $X$, 
we have $a\us^{[n]}b=t^na+(s^n-t^n)b$ 
and $a\os^{[n]}b=s^na$ 
for any $a,b \in X$.

We define the \emph{type} of a biquandle $X$ by 
\begin{align*}
\type X=\min \{ n>0 \mid a\us^{[n]}b=a=a\os^{[n]}b ~(\forall a,b \in X) \}.
\end{align*}
Any finite biquandle is of finite type \cite{IN17}.

We also recall the definitions of a $G$-family of quandles 
and a $G$-family of biquandles.

\begin{definition}[\cite{IIJO13}]
Let $G$ be a group with the identity element $e$.
A $G$-family of quandles is a non-empty set $X$ 
with a family of binary operations $*^g:X \times X \to X ~(g \in G)$ 
satisfying the following axioms.
\begin{itemize}
\item
For any $x \in X$ and $g \in G$, 
$x*^gx=x.$
\item
For any $x,y \in X$ and $g,h \in G$, 
$x*^{gh}y=(x*^gy)*^hy$ and $x*^ey=x$.
\item\
For any $x,y,z \in X$ and $g,h \in G$, 
$(x*^gy)*^hz=(x*^hz)*^{h^{-1}gh}(y*^hz)$.
\end{itemize}
\end{definition}

\begin{definition}[\cite{IIKKMOpre,IN17}]
Let $G$ be a group with the identity element $e$.
A $G$-family of biquandles is a non-empty set $X$ 
with two families of binary operations $\us^g, \os^g :X \times X \to X ~(g \in G)$ 
satisfying the following axioms.
\begin{itemize}
\item
For any $x \in X$ and $g \in G$, 
\begin{align*}
x\us^gx=x\os^gx.
\end{align*}
\item
For any $x,y \in X$ and $g,h \in G$, 
\begin{align*}
& x\us^{gh}y=(x\us^gy)\us^h(y\us^gy),~~ x\us^ey=x,\\
& x\os^{gh}y=(x\os^gy)\os^h(y\os^gy),~~ x\os^ey=x.
\end{align*}
\item
For any $x,y,z \in X$ and $g,h \in G$, 
\begin{align*}
(x\us^g y)\us^h(z\os^gy)=(x\us^hz)\us^{h^{-1}gh}(y\us^hz),\\
(x\os^g y)\us^h(z\os^gy)=(x\us^hz)\os^{h^{-1}gh}(y\us^hz),\\
(x\os^g y)\os^h(z\os^gy)=(x\os^hz)\os^{h^{-1}gh}(y\us^hz).
\end{align*}
\end{itemize}
\end{definition}

For a biquandle $(X,\us,\os)$ with $\type X < \infty$, 
$(X,\{ \us^{[n]} \}_{[n] \in \mathbb{Z}_{\type X}},\{ \os^{[n]} \}_{[n] \in \mathbb{Z}_{\type X}})$ 
is a $\mathbb{Z}_{\type X}$-family of biquandles \cite{IN17}.
In particular, 
when $X$ is an Alexander biquandle, 
$(X,\{ \us^{[n]} \}_{[n] \in \mathbb{Z}_{\type X}},\{ \os^{[n]} \}_{[n] \in \mathbb{Z}_{\type X}})$ 
is called a \emph{$\mathbb{Z}_{\type X}$-family of Alexander biquandles}.

\section{Colorings}
In this section, 
we introduce a coloring of a diagram of an $S^1$-oriented handlebody-link 
by a $G$-family of biquandles.
Let $G$ be a group 
and let $D$ be a diagram of an $S^1$-oriented handlebody-link $H$.
A \emph{$G$-flow} of $D$ is a map $\phi : \mathcal{A}(D) \to G$ satisfying

\begin{center}
\includegraphics[width=85mm]{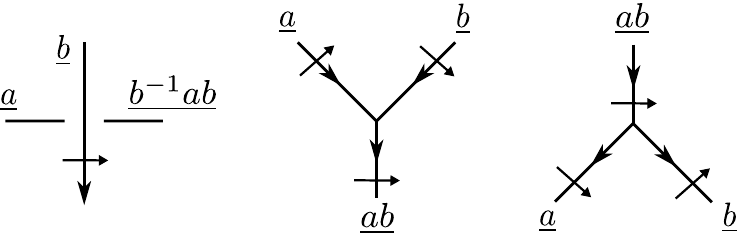}
\end{center}

\noindent
at each crossing and each vertex.
In this paper, 
to avoid confusion, 
we often represent an element of $G$ with an underline.
We denote by $(D,\phi)$, 
which is called a \emph{$G$-flowed diagram} of $H$, 
a diagram $D$ given a $G$-flow $\phi$ 
and put 
$\flow(D;G):=\{ \phi \mid \phi : G \textrm{-flow~of~} D \}$.
We can identify a $G$-flow $\phi$ 
with a homomorphism from the fundamental group $\pi_1(S^3-H)$ to $G$.

Let $G$ be a group 
and let $D$ be a diagram of an $S^1$-oriented handlebody-link $H$.
Let $D'$ be a diagram obtained by applying one of Reidemeister moves 
to the diagram $D$ once.
For any $G$-flow $\phi$ of $D$, 
there is an unique $G$-flow $\phi'$ of $D'$ 
which coincides with $\phi$ 
except near the point where the move applied.
Therefore 
the number of $G$-flow of $D$, 
denoted by $\# \flow(D;G)$, 
is an invariant of $H$.
We call the $G$-flow $\phi'$ 
the \emph{associated $G$-flow} of $\phi$ 
and the $G$-flowed diagram $(D',\phi')$ 
the \emph{associated $G$-flowed diagram} of $(D,\phi)$.

For any $m \in \mathbb{Z}_{\geq 0}$ 
and $\mathbb{Z}_m$-flow $\phi$ of a diagram $D$ of an $S^1$-oriented handlebody-link $H$, 
we define 
$\gcd \phi := \gcd \{ \phi (a), m \mid a \in \mathcal{A}(D) \}$.
Then we have the following lemma 
in the same way as in \cite{IK10}.

\begin{lemma}\label{gcd}
For any $m \in \mathbb{Z}_{\geq 0}$, 
let $(D,\phi)$ be a $\mathbb{Z}_m$-flowed diagram of an $S^1$-oriented handlebody-link H 
and let $(D',\phi')$ be the associated $\mathbb{Z}_m$-flowed diagram of $(D,\phi)$.
Then it follows that 
$\gcd \phi=\gcd \phi'$.
\end{lemma}

Let $G$ be a group, $X$ be a $G$-family of biquandles 
and let $(D,\phi)$ be a $G$-flowed diagram of an $S^1$-oriented handlebody-link $H$.
An \emph{$X$-coloring} of $(D,\phi)$ is a map 
$C : \mathcal{SA}(D,\phi) \to X$ satisfying

\begin{center}
\includegraphics[width=130mm]{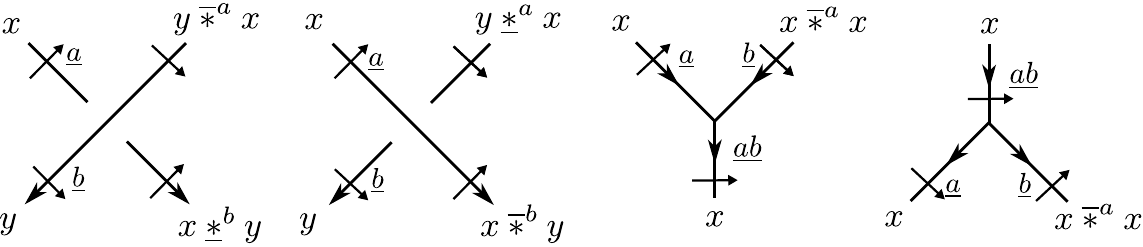}
\end{center}

\noindent
at each crossing and each vertex, 
where $\mathcal{SA}(D,\phi)$ is the set of all semi-arcs of $(D,\phi)$.
We denote by $\col_X(D,\phi)$ 
the set of all $X$-colorings of $(D,\phi)$.
We note that 
$\col_X(D,\phi)$ is a vector space over $X$ 
when $X$ is a field.

\begin{proposition}[\cite{IN17}]\label{IN17}
Let $X$ be a $G$-family of biquandles  
and let $(D,\phi)$ be a $G$-flowed diagram of an $S^1$-oriented handlebody-link H.
Let $(D',\phi')$ be the associated $G$-flowed diagram of $(D,\phi)$.
For any $X$-coloring $C$ of $(D,\phi)$, 
there is an unique $X$-coloring $C'$ of $(D',\phi')$ 
which coincides with $C$ 
except near the point where the move applied. 
\end{proposition}

We call the $X$-coloring $C'$ 
the \emph{associated $X$-coloring} of $C$.
By this proposition, 
we have $\# \col_X(D,\phi) = \# \col_X(D',\phi')$.

\begin{proposition}\label{trivial col.}
Let $G$ be a group 
and let X be a $G$-family of biquandles.
Then the following hold.

\begin{enumerate}
\item
Let $(D,\phi)$ be a $G$-flowed diagram of an $S^1$-oriented handlebody-link.
Then it follows that $\#\col_X(D,\phi) \geq \#X$.
\item
Let $(O,\psi)$ be a $G$-flowed diagram of an $S^1$-oriented 
$m$-component trivial handlebody-link.
Then it follows that $\#\col_X(O,\psi) = (\#X)^m$.
\end{enumerate}
\end{proposition}

\begin{proof} 
\begin{enumerate}
\item
By Theorem \ref{Reidemeister moves} and \cite{Murao16}, 
we can deform $(D,\phi)$ into the $G$-flowed diagram $(D',\phi')$
depicted in Figure \ref{bind pres.} 
by a finite sequence of Reidemeinter moves preserving Y-orientations, 
where $b$ is a classical $l$-braid, 
and $a_{i,1}, \ldots, a_{i,m_i},b_{i,1}, \ldots, b_{i,n_i} \in G$ for any $i=1,\ldots,s$.
We note that 
$\prod_{j=1}^{m_i}a_{i,j}=\prod_{j=1}^{n_i}b_{i,j}$ for any $i=1,\ldots,s$, 
and $x \us^g x=x \os^g x$ for any $x \in X$ and $g \in G$.
By Proposition \ref{IN17}, 
it is sufficient to prove that 
$\#\col_X(D',\phi') \geq \#X$.
Here for any $x \in X$ and $g \in G$, 
we write $x*^gx$ for $x \us^g x$ and $x \os^g x$ simply.
Then for any $x \in X$, 
the assignment of elements of $X$ to each semi-arc of $(D',\phi')$ 
as shown in Figures \ref{bind pres.}  and \ref{braid col.} 
is an $X$-coloring, 
where each $g_i$ represents an element of $G$ 
in Figure \ref{braid col.}.
Therefore we have $\#\col_X(D',\phi') \geq \#X$.

\begin{figure}[htb]
\begin{center}
\includegraphics[width=134mm]{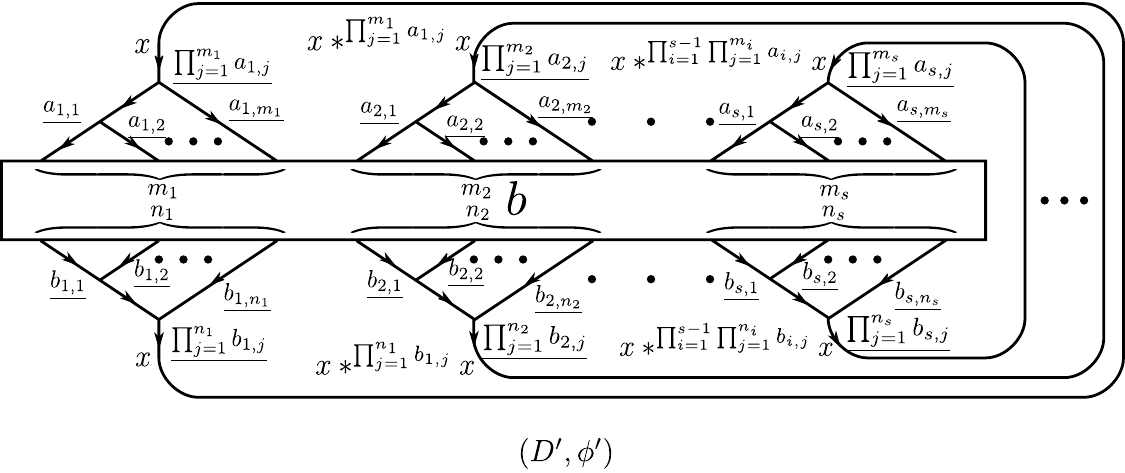}
\end{center}
\caption{A $G$-flowed diagram $(D',\phi')$ and its $X$-coloring.}\label{bind pres.} 
\end{figure}

\begin{figure}[htb]
\begin{center}
\includegraphics[width=134mm]{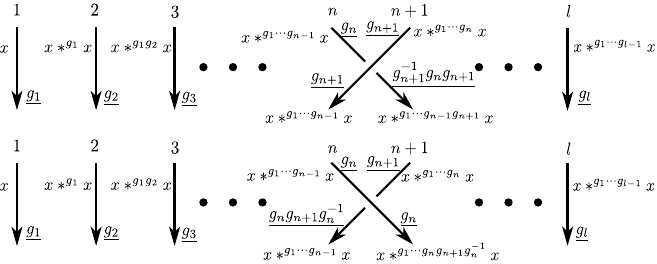}
\end{center}
\caption{An $X$-coloring of $(D',\phi')$ in the part of $b$.}\label{braid col.} 
\end{figure}

\item
It is sufficient to prove that 
$\#\col_X(O,\psi) = \#X$ 
when $m=1$.
Let $(O_g,\psi_g)$ be a $G$-flowed diagram of 
an $S^1$-oriented trivial handlebody-knot of genus $g$.
By Theorem \ref{Reidemeister moves}, 
we can deform $(O_g,\psi_g)$ into the $G$-flowed diagram $(O'_g,\psi'_g)$
depicted in Figure \ref{trivial hdbdy-knot col.} 
by a finite sequence of Reidemeinter moves preserving Y-orientations, 
where $a_i \in G$ for any $i=1,\ldots,g$, and $e$ is the identity of $G$.
By Proposition \ref{IN17}, 
it is sufficient to prove that 
$\#\col_X(O'_g,\psi'_g) = \#X$.
For any $x \in X$, 
the assignment of $x$ to each semi-arc of $(O'_g,\psi'_g)$ 
as shown in Figure \ref{trivial hdbdy-knot col.} 
is an $X$-coloring.
On the other hand, 
since any $X$-coloring of $(O_g',\psi'_g)$ is given 
by Figure \ref{trivial hdbdy-knot col.} 
for some $x \in X$, 
we have $\#\col_X(O_g',\psi'_g)=\#X$.
\end{enumerate}
\end{proof}

\begin{figure}[htb]
\begin{center}
\includegraphics[width=80mm]{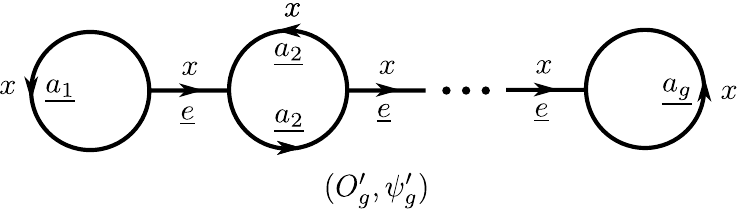}
\end{center}
\caption{A $G$-flowed diagram $(O'_g,\psi'_g)$ and its $X$-coloring.}\label{trivial hdbdy-knot col.} 
\end{figure}

\section{Linear relationships for Alexander biquandle colorings}

For any $\mathbb{Z}_m$-flowed diagram $(D,\phi)$ of an $S^1$-oriented handlebody-link, 
we define the \emph{Alexander numbering} of $(D,\phi)$ 
by assigning elements of $\mathbb{Z}_m$ 
to each region of $(D,\phi)$ 
as shown in Figure \ref{Alex. numbering}, 
where the unbounded region is labeled $0$.
It is an extension of the Alexander numbering of a classical knot diagram \cite{Alexander23}.
It is easy to see that 
for any $\mathbb{Z}_m$-flowed diagram $(D,\phi)$ of an $S^1$-oriented handlebody-link, 
there uniquely exists the Alexander numbering of $(D,\phi)$.
For example, 
a $\mathbb{Z}_m$-flowed diagram of the handlebody-knot $5_2$ \cite{IKMS12} 
with the Alexander numbering 
is depicted in Figure \ref{ex. Alex. numbering}.
For any semi-arc $\alpha$ of $(D,\phi)$, 
we denote by $\rho (\alpha)$ 
the Alexander number of the region 
which the normal orientation of $\alpha$ points to.

\begin{figure}[htb]
\begin{center}
\includegraphics[width=25mm]{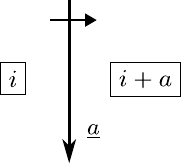}
\end{center}
\caption{The Alexander numbering of $(D,\phi)$.}\label{Alex. numbering} 
\end{figure}

\begin{figure}[htb]
\begin{center}
\includegraphics[width=50mm]{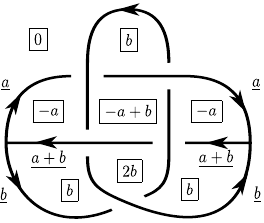}
\end{center}
\caption{A $\mathbb{Z}_m$-flowed diagram of $5_2$ with the Alexander numbering.}\label{ex. Alex. numbering} 
\end{figure}

In the following, 
every component of a diagram of any $S^1$-oriented handlebody-link 
may have a crossing at least $1$.
Let $(D,\phi)$ be a $\mathbb{Z}_m$-flowed diagram of an $S^1$-oriented handlebody-link 
with the Alexander numbering 
and let $X$ be a $\mathbb{Z}_m$-family of Alexander biquandles.
We put $C(D,\phi)=\{ c_1,\ldots,c_n \}$ 
and $V(D,\phi)=\{ \tau_1,\ldots ,\tau_{2k} \}$, 
where $C(D,\phi)$ and $V(D,\phi)$ are 
the set of all crossings of $(D,\phi)$ and
the one of all  vertices of $(D,\phi)$ respectively, 
where the sign of $\tau_i$ is $1$ for any $i=1,\ldots,k$ 
and $-1$ for any $i=k+1,\ldots,2k$.
Then we denote by $x_i$ 
each semi-arc of $(D,\phi)$ as shown in Figure \ref{semi-arcs}, 
which implies 
$\mathcal{SA}(D,\phi)=\{x_1, \ldots ,x_{2n+3k}\}$.

\begin{figure}[htb]
\begin{center}
\includegraphics[width=120mm]{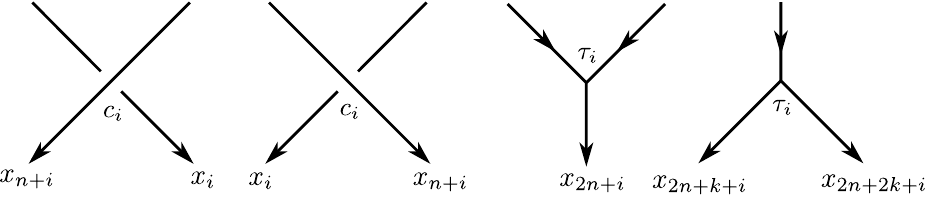}
\end{center}
\caption{Semi-arcs $x_i$ of $(D,\phi).$}\label{semi-arcs} 
\end{figure}

We denote by $u_i$, $v_i$, $v'_i$, $w_i$, $\alpha_i$, $\beta_i$ and $\gamma_i$ 
the semi-arcs incident to a crossing $c_i$ or a vertex $\tau_i$ 
as shown in Figure \ref{notation}.
We put $\phi_i:=\phi(u_i)=\phi(w_i)$, 
$\psi_i:=\phi(v_i)=\phi(v'_i)$, 
$\eta_i:=\phi(\alpha_i)$ and 
$\theta_i:=\phi(\beta_i)$.
We denote by $\epsilon_{c_i} \in \{\pm1 \}$ and $\epsilon_{\tau_i} \in \{\pm1 \}$ 
the signs of a crossing $c_i$ 
and a vertex $\tau_i$ respectively (see Figure \ref{notation}).

\begin{figure}[htb]
\begin{center}
\includegraphics[width=120mm]{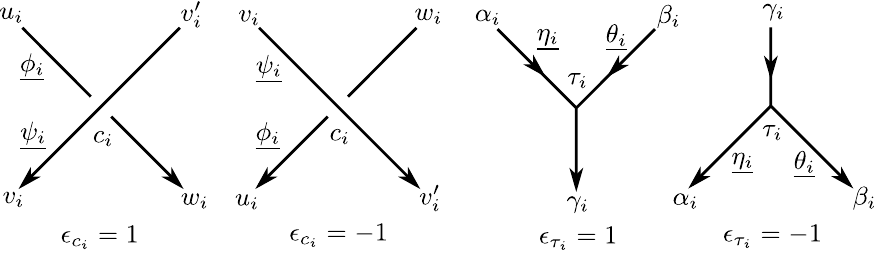}
\end{center}
\caption{Notations.}\label{notation} 
\end{figure}

For any semi-arcs $y,y' \in \mathcal{SA}(D,\phi)$, 
we put 
\begin{align*}
\delta (y,y'):=
\begin{cases}
1 & (y=y'),\\
0 & (y \neq y').
\end{cases}
\end{align*}
Then we define a matrix 
$A(D,\phi;X)= (a_{i,j}) \in M(2n+4k,2n+3k;X)$ by 
\begin{align*}
a_{i,j}=
\begin{cases}
\delta(u_i,x_j)t^{\psi_i}+\delta(v_i,x_j)(s^{\psi_i}-t^{\psi_i})-\delta(w_i,x_j) & (1 \leq i \leq n),\\
-\delta(v_{i-n},x_j)s^{\phi_{i-n}}+\delta(v'_{i-n},x_j) & (n+1 \leq i \leq 2n),\\
\delta(\alpha_{i-2n},x_j)-\delta(\gamma_{i-2n},x_j) & (2n+1 \leq i \leq 2n+2k),\\
\delta(\beta_{i-2n-2k},x_j)-\delta(\gamma_{i-2n-2k},x_j)s^{\eta_{i-2n-2k}} & (2n+2k+1 \leq i \leq 2n+4k).
\end{cases}
\end{align*}
We note that 
$A(D,\phi;X)$ is ditermined 
up to permuting of rows and columns of the matrix, 
and it follows that 
\begin{align*}
\col_X(D,\phi)=
\left\{ 
\begin{pmatrix}
z_1\\
z_2\\
\vdots\\
z_{2n+3k}
\end{pmatrix}
\in X^{2n+3k}
\middle |
A(D,\phi;X)
\begin{pmatrix}
z_1\\
z_2\\
\vdots\\
z_{2n+3k}
\end{pmatrix}
=\bm{0}
\right\}.
\end{align*}

For example, 
let $(E,\psi)$ be the $\mathbb{Z}_m$-flowed diagram 
of the handlebody-knot depicted in Figure \ref{ex. mtx.}.
Then we have 
\begin{align*}
A(E,\psi;X)=
\begin{pmatrix}
-1 & 0 & s^a-t^a & t^a & 0 & 0 & 0 \\
0 & -1 & 0 & s^b-t^b & 0 & t^b & 0 \\
0 & 1 & -s^b & 0 & 0 & 0 & 0 \\
0 & 0 & 0 & -s^a & 0 & 0 & 1 \\
0 & 0 & 1 & 0 & -1 & 0 & 0 \\
0 & 0 & 0 & 0 & -1 & 1 & 0 \\
1 & 0 & 0 & 0 & -s^a & 0 & 0 \\
0 & 0 & 0 & 0 & -s^a & 0 & 1 \\
\end{pmatrix}.
\end{align*}

\begin{figure}[htb]
\begin{center}
\includegraphics[width=50mm]{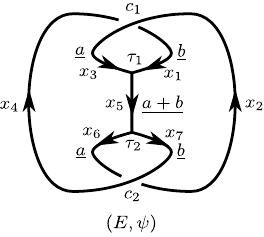}
\end{center}
\caption{A $\mathbb{Z}_m$-flowed diagram $(E,\psi)$.}\label{ex. mtx.} 
\end{figure}

Then we have the following proposition.

\begin{proposition}\label{linearity}
Let $(D,\phi)$ be a $\mathbb{Z}_m$-flowed diagram of an $S^1$-oriented handlebody-link 
with the Alexander numbering 
and let $X$ be a $\mathbb{Z}_m$-family of Alexander biquandles.
Let $\bm{a}_i$ be the $i$-th row of $A(D,\phi;X)$, 
that is,
\begin{align*}
A(D,\phi;X)=(a_{i,j})=
\begin{pmatrix}
\bm{a}_1\\
\bm{a}_2\\
\vdots \\
\bm{a}_{2n+4k}
\end{pmatrix}.
\end{align*}
Then it follows that 
\begin{align*}
& \sum_{i=1}^{n}\epsilon_{c_i}t^{-\rho(w_i)}(s^{\phi_i}-t^{\phi_i})\bm{a}_i
+\sum_{i=1}^{n}\epsilon_{c_i}t^{-\rho(v'_i)}(s^{\psi_i}-t^{\psi_i})\bm{a}_{n+i}\\
&+\sum_{i=1}^{2k}\epsilon_{\tau_i}t^{-\rho(\alpha_i)}(s^{\eta_i}-t^{\eta_i})\bm{a}_{2n+i}
+\sum_{i=1}^{2k}\epsilon_{\tau_i}t^{-\rho(\beta_i)}(s^{\theta_i}-t^{\theta_i})\bm{a}_{2n+2k+i}=\bm{0}.
\end{align*}
\end{proposition}

\begin{proof}
For any semi-arc $y$ incident to a crossing or a vertex $\sigma$, 
we put
\begin{align*}
\epsilon (y;\sigma):=
\begin{cases}
1 & \textrm{if the orientation of $y$ points to $\sigma$},\\
-1 & \textrm{otherwise}.
\end{cases}
\end{align*}

It is sufficient to prove that for any $j=1,2,\ldots,2n+3k$, 
\begin{align*}
& \sum_{i=1}^{n}\epsilon_{c_i}t^{-\rho(w_i)}(s^{\phi_i}-t^{\phi_i})a_{i,j}
+\sum_{i=1}^{n}\epsilon_{c_i}t^{-\rho(v'_i)}(s^{\psi_i}-t^{\psi_i})a_{n+i,j}\\
&+\sum_{i=1}^{2k}\epsilon_{\tau_i}t^{-\rho(\alpha_i)}(s^{\eta_i}-t^{\eta_i})a_{2n+i,j}
+\sum_{i=1}^{2k}\epsilon_{\tau_i}t^{-\rho(\beta_i)}(s^{\theta_i}-t^{\theta_i})a_{2n+2k+i,j}=0.
\end{align*}
For the first term, we have 
\begin{align*}
& \epsilon_{c_i}t^{-\rho(w_i)}(s^{\phi_i}-t^{\phi_i})\delta(u_i,x_j)t^{\psi_i}
= \delta(u_i,x_j) \epsilon(u_i;c_i)t^{-\rho(u_i)}(s^{\phi(u_i)}-t^{\phi(u_i)}), \\
& \epsilon_{c_i}t^{-\rho(w_i)}(s^{\phi_i}-t^{\phi_i})\delta(v_i,x_j)(s^{\psi_i}-t^{\psi_i}) \tag{1}\\
&= \epsilon_{c_i}t^{-\rho(w_i)}s^{\phi_i}\delta(v_i,x_j)(s^{\psi_i}-t^{\psi_i})
-\epsilon_{c_i}t^{-\rho(w_i)}t^{\phi_i}\delta(v_i,x_j)(s^{\psi_i}-t^{\psi_i})\\
&= \epsilon_{c_i}t^{-\rho(w_i)}\delta(v_i,x_j)(s^{\psi_i}-t^{\psi_i})s^{\phi_i}
+\delta(v_i,x_j)\epsilon(v_i;c_i)t^{-\rho(v_i)}(s^{\phi(v_i)}-t^{\phi(v_i)}),\\
& \epsilon_{c_i}t^{-\rho(w_i)}(s^{\phi_i}-t^{\phi_i})(-\delta(w_i,x_j))
= \delta(w_i,x_j) \epsilon(w_i;c_i)t^{-\rho(w_i)}(s^{\phi(w_i)}-t^{\phi(w_i)}). 
\end{align*}
For the second term, we have 
\begin{align*}
& \epsilon_{c_i}t^{-\rho(v'_i)}(s^{\psi_i}-t^{\psi_i})(-\delta(v_i,x_j)s^{\phi_i})
= -\epsilon_{c_i}t^{-\rho(v'_i)}\delta(v_i,x_j)(s^{\psi_i}-t^{\psi_i})s^{\phi_i}, \tag{2} \\
& \epsilon_{c_i}t^{-\rho(v'_i)}(s^{\psi_i}-t^{\psi_i})\delta(v'_i,x_j)
=\delta(v'_i,x_j) \epsilon(v'_i;c_i)t^{-\rho(v'_i)}(s^{\phi(v'_i)}-t^{\phi(v'_i)}).
\end{align*}
For the third term, we have 
\begin{align*}
& \epsilon_{\tau_i}t^{-\rho(\alpha_i)}(s^{\eta_i}-t^{\eta_i})\delta(\alpha_i,x_j)
=\delta(\alpha_i,x_j) \epsilon(\alpha_i;\tau_i)t^{-\rho(\alpha_i)}(s^{\phi(\alpha_i)}-t^{\phi(\alpha_i)}), \\
& \epsilon_{\tau_i}t^{-\rho(\alpha_i)}(s^{\eta_i}-t^{\eta_i})(-\delta(\gamma_i,x_j))
=\delta(\gamma_i,x_j) \epsilon(\gamma_i;\tau_i)t^{-\rho(\gamma_i)}t^{\theta_i}(s^{\eta_i}-t^{\eta_i}). \tag{3}
\end{align*}
For the last term, we have 
\begin{align*}
& \epsilon_{\tau_i}t^{-\rho(\beta_i)}(s^{\theta_i}-t^{\theta_i})\delta(\beta_i,x_j)
=\delta(\beta_i,x_j) \epsilon(\beta_i;\tau_i)t^{-\rho(\beta_i)}(s^{\phi(\beta_i)}-t^{\phi(\beta_i)}),  \\
&\epsilon_{\tau_i}t^{-\rho(\beta_i)}(s^{\theta_i}-t^{\theta_i})(-\delta(\gamma_i,x_j)s^{\eta_i})
=\delta(\gamma_i,x_j)\epsilon(\gamma_i;\tau_i)t^{-\rho(\gamma_i)}(s^{\theta_i}-t^{\theta_i})s^{\eta_i}. \tag{4}
\end{align*}
We note that 
\begin{align*}
& (1)+(2)
=\delta(v_i,x_j)\epsilon(v_i;c_i)t^{-\rho(v_i)}(s^{\phi(v_i)}-t^{\phi(v_i)}), \\
& (3)+(4)
=\delta(\gamma_i,x_j)\epsilon(\gamma_i;\tau_i)t^{-\rho(\gamma_i)}(s^{\phi(\gamma_i)}-t^{\phi(\gamma_i)}).
\end{align*}
Therefore for any $j=1,2,\ldots,2n+3k$, it follows that 
\begin{align*}
& \sum_{i=1}^{n}\epsilon_{c_i}t^{-\rho(w_i)}(s^{\phi_i}-t^{\phi_i})a_{i,j}
+\sum_{i=1}^{n}\epsilon_{c_i}t^{-\rho(v'_i)}(s^{\psi_i}-t^{\psi_i})a_{n+i,j}\\
& \quad +\sum_{i=1}^{2k}\epsilon_{\tau_i}t^{-\rho(\alpha_i)}(s^{\eta_i}-t^{\eta_i})a_{2n+i,j}
+\sum_{i=1}^{2k}\epsilon_{\tau_i}t^{-\rho(\beta_i)}(s^{\theta_i}-t^{\theta_i})a_{2n+2k+i,j}\\
&=\sum_{i=1}^{n}(\delta(u_i,x_j) \epsilon(u_i;c_i)t^{-\rho(u_i)}(s^{\phi(u_i)}-t^{\phi(u_i)})\\
& \qquad +\delta(v_i,x_j) \epsilon(v_i;c_i)t^{-\rho(v_i)}(s^{\phi(v_i)}-t^{\phi(v_i)})\\
& \qquad +\delta(v'_i,x_j) \epsilon(v'_i;c_i)t^{-\rho(v'_i)}(s^{\phi(v'_i)}-t^{\phi(v'_i)})\\
& \qquad +\delta(w_i,x_j) \epsilon(w_i;c_i)t^{-\rho(w_i)}(s^{\phi(w_i)}-t^{\phi(w_i)}))\\
& \quad + \sum_{i=1}^{2k}(\delta(\alpha_i,x_j)\epsilon(\alpha_i;\tau_i)t^{-\rho(\alpha_i)}(s^{\phi(\alpha_i)}-t^{\phi(\alpha_i)})\\
& \qquad  \quad +\delta(\beta_i,x_j)\epsilon(\beta_i;\tau_i)t^{-\rho(\beta_i)}(s^{\phi(\beta_i)}-t^{\phi(\beta_i)})\\
& \qquad  \quad +\delta(\gamma_i,x_j)\epsilon(\gamma_i;\tau_i)t^{-\rho(\gamma_i)}(s^{\phi(\gamma_i)}-t^{\phi(\gamma_i)}))\\
&= t^{-\rho(x_j)}(s^{\phi(x_j)}-t^{\phi(x_j)})-t^{-\rho(x_j)}(s^{\phi(x_j)}-t^{\phi(x_j)})\\
&= 0.
\end{align*}
\end{proof}

Let $X$ be an Alexander biquandle and let $m=\type X$.
Then $X$ is also a $\mathbb{Z}_m$-family of Alexander biquandles.
Let $D$ be an oriented classical link diagram.
We can regard $D$ as a $\mathbb{Z}_m$-flowed diagram $(D,\phi_{(1)})$ 
of an $S^1$-oriented handlebody-link 
whose components are of genus $1$, 
where $\phi_{(1)}$ is the constant map to $1$.
Hence we can regard an $X$-coloring of $D$ 
as an $X$-coloring of $(D,\phi_{(1)})$.
We define a matrix $A(D;X) \in M(2n,2n;X)$ by 
$A(D;X)=A(D,\phi_{(1)};X)$, 
where $n$ is the number of crossings of $D$.
Then the set of all $X$-colorings of $D$, 
denoted by $\col_X(D)$, 
is given by 
\begin{align*}
\col_X(D)=
\left\{ 
\begin{pmatrix}
z_1\\
z_2\\
\vdots\\
z_{2n}
\end{pmatrix}
\in X^{2n}
\middle |
A(D;X)
\begin{pmatrix}
z_1\\
z_2\\
\vdots\\
z_{2n}
\end{pmatrix}
=\bm{0}
\right\}.
\end{align*}
Therefore we obtain the following corollary.

\begin{corollary}\label{linear relationship}
Let $D$ be a diagram of an oriented classical link with the Alexander numbering 
and let $X$ be an Alexander biquandle.
Let $\bm{a}_i$ be the $i$-th row of $A(D;X)$, 
that is,
\begin{align*}
A(D;X)=(a_{i,j})=
\begin{pmatrix}
\bm{a}_1\\
\bm{a}_2\\
\vdots \\
\bm{a}_{2n}
\end{pmatrix}.
\end{align*}
Then it follows that 
\begin{align*}
& \sum_{i=1}^{n}\epsilon_{c_i}t^{-\rho(w_i)}(s-t)\bm{a}_i
+\sum_{i=1}^{n}\epsilon_{c_i}t^{-\rho(v'_i)}(s-t)\bm{a}_{n+i}=\bm{0}.
\end{align*}
\end{corollary}

\section{Main theorem}

In this section, 
we give lower bounds for the Gordian distance and the unknotting number 
of $S^1$-oriented handlebody-knots.

\begin{theorem}\label{gordian distance}
Let $H_i$ be an $S^1$-oriented handlebody-knot of genus $g$  
and let $D_i$ be a diagram of $H_i$ $(i=1,2)$.
Let $X=\mathbb{Z}_p[t^{\pm 1}]/(f(t))$ 
which is a $\mathbb{Z}_m$-family of Alexander biquandles, 
where $p$ is a prime number, $s \in \mathbb{Z}_p[t^{\pm 1}]$ 
and $f(t) \in \mathbb{Z}_p[t^{\pm 1}]$ is an irreducible polynomial.
Then it follows that
\begin{align*}
\max_{\phi_1 \in \flow (D_1;\mathbb{Z}_m)} 
\min_{\substack{\phi_2 \in \flow (D_2;\mathbb{Z}_m) \\ \gcd \phi_1=\gcd \phi_2}}
|\dim \col_X(D_1,\phi_1)-\dim \col_X(D_2,\phi_2)| \leq d(H_1,H_2).
\end{align*}
\end{theorem}

\begin{proof}
Let $(D,\phi)$ be a $\mathbb{Z}_m$-flowed diagram 
of an $S^1$-oriented handlebody-knot 
and let  $C(D,\phi)=\{ c_1,\ldots,c_n \}$ 
and $V(D,\phi)=\{ \tau_1,\ldots ,\tau_{2k} \}$.
Let $(\ol{D},\ol{\phi})$ be the $\mathbb{Z}_m$-flowed diagram 
of an $S^1$-oriented handlebody-knot 
which is obtained from $(D,\phi)$ by the crossing change at $c_1$ 
and let $C(\ol{D},\ol{\phi})=\{ \ol{c}_1,\ldots, \ol{c}_n \}$ 
and $V(\ol{D},\ol{\phi})=\{ \ol{\tau}_1,\ldots ,\ol{\tau}_{2k} \}$, 
where $\ol{\phi}$, $\ol{c}_i$ and $\ol{\tau}_i$ originate from 
$\phi$, $c_i$ and $\tau_i$ naturally and respectively 
(see Figure \ref{col. crossing change}).
In the following, 
we show that 
\[ |\dim \col_X(D,\phi)-\dim \col_X(\ol{D},\ol{\phi})| \leq 1,\]
that is, 
\[ |\rank A(D,\phi;X)-\rank A(\ol{D},\ol{\phi};X)| \leq 1.\]

\begin{figure}[htb]
\begin{center}
\includegraphics[width=100mm]{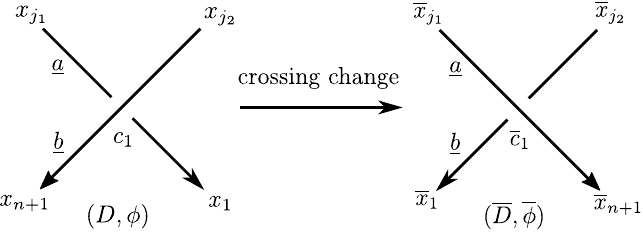}
\end{center}
\caption{The crossing change at $c_1$.}\label{col. crossing change}
\end{figure}

We may assume that $c_1$ is a positive crossing 
and $\ol{c}_1$ is a negative crossing.
We denote by $\ol{x}_i$ each semi-arc of $(\ol{D},\ol{\phi})$  
in the same way as in Figure \ref{semi-arcs} 
with respect to $\ol{c}_i$ or $\ol{\tau}_i$, 
and so are 
$\ol{v}'_i$, $\ol{w}_i$, 
$\ol{\alpha}_i$, $\ol{\beta}_i$, 
$\ol{\phi}_i$, $\ol{\psi}_i$, $\ol{\eta}_i$, $\ol{\theta}_i$, 
$\ol{\epsilon}_{c_i}$ and $\ol{\epsilon}_{\tau_i}$ 
(see Figure \ref{notation}).
We denote by $x_{j_1}$ and $x_{j_2}$ 
the semi-arcs which point to the crossing $c_1$ 
of $(D,\phi)$ 
as shown in Figure \ref{col. crossing change}, 
and we put 
$a := \phi_1=\ol{\psi}_1$ and $b := \psi_1=\ol{\phi}_1$.
We note that 
$\col_X(D,\phi)$ and $\col_X(\ol{D},\ol{\phi})$ are vector spaces over $X$
since $X$ is a field.

Let 
\begin{align*}
A(D,\phi;X)=(a_{i,j})=
\begin{pmatrix}
\bm{a}_1\\
\bm{a}_2\\
\vdots \\
\bm{a}_{2n+4k}
\end{pmatrix},~
A(\ol{D},\ol{\phi};X)=(\ol{a}_{i,j})=
\begin{pmatrix}
\bm{\ol{a}}_1\\
\bm{\ol{a}}_2\\
\vdots \\
\bm{\ol{a}}_{2n+4k}
\end{pmatrix}
\end{align*}
and let 

\begin{align*}
\hat{A}(\ol{D},\ol{\phi};X)=(\hat{a}_{i,j})=
\begin{pmatrix}
\bm{\hat{a}}_1\\
\bm{\hat{a}}_2\\
\vdots \\
\bm{\hat{a}}_{2n+4k}
\end{pmatrix},
\end{align*}

\noindent
where $\bm{\hat{a}}_i$ is a vector 
obtained by permuting the first entry and the $(n+1)$-th entry of $\bm{\ol{a}}_i$.
Then we have $\bm{a}_i=\bm{\hat{a}}_i$ when $i \neq 1,n+1$.
We note that $\rank A(\ol{D},\ol{\phi};X)=\rank \hat{A}(\ol{D},\ol{\phi};X)$ and 

\begin{align*}
\bm{a}_1 &= (-1,0,\ldots,0,\stackrel{j_1}{\stackrel{\vee}{t^b}},0,\ldots,0,\stackrel{n+1}{\stackrel{\vee}{s^b-t^b}},0,\ldots,0),\\
\bm{a}_{n+1} &=  (0,\ldots,0,\stackrel{j_2}{\stackrel{\vee}{1}},0,\ldots,0,\stackrel{n+1}{\stackrel{\vee}{-s^a}},0,\ldots,0),\\
\bm{\ol{a}}_1 &= (t^a,0,\ldots,0,\stackrel{j_1}{\stackrel{\vee}{s^a-t^a}},0,\ldots,0,\stackrel{j_2}{\stackrel{\vee}{-1}},0,\ldots,0),\\
\bm{\ol{a}}_{n+1} &= (0,\ldots,0,\stackrel{j_1}{\stackrel{\vee}{-s^b}},0,\ldots,0,\stackrel{n+1}{\stackrel{\vee}{1}},0,\ldots,0),\\
\bm{\hat{a}}_1 &= (0,\ldots,0,\stackrel{j_1}{\stackrel{\vee}{s^a-t^a}},0,\ldots,0,\stackrel{j_2}{\stackrel{\vee}{-1}},0,\ldots,0,\stackrel{n+1}{\stackrel{\vee}{t^a}},0,\ldots,0),\\
\bm{\hat{a}}_{n+1} &= (1,0,\ldots,0,\stackrel{j_1}{\stackrel{\vee}{-s^b}},0,\ldots,0).
\end{align*}

\noindent
By Proposition \ref{linearity}, 
we obtain 
\begin{align*}
& \sum_{i=1}^{n}\epsilon_{c_i}t^{-\rho(w_i)}(s^{\phi_i}-t^{\phi_i})\bm{a}_i
+\sum_{i=1}^{n}\epsilon_{c_i}t^{-\rho(v'_i)}(s^{\psi_i}-t^{\psi_i})\bm{a}_{n+i}\\
&+\sum_{i=1}^{2k}\epsilon_{\tau_i}t^{-\rho(\alpha_i)}(s^{\eta_i}-t^{\eta_i})\bm{a}_{2n+i}
+\sum_{i=1}^{2k}\epsilon_{\tau_i}t^{-\rho(\beta_i)}(s^{\theta_i}-t^{\theta_i})\bm{a}_{2n+2k+i}=\bm{0}
\end{align*}
and  
\begin{align*}
& \sum_{i=1}^{n}\ol{\epsilon}_{c_i}t^{-\rho(\ol{w}_i)}(s^{\ol{\phi}_i}-t^{\ol{\phi}_i})\bm{\ol{a}}_i
+\sum_{i=1}^{n}\ol{\epsilon}_{c_i}t^{-\rho(\ol{v}'_i)}(s^{\ol{\psi}_i}-t^{\ol{\psi}_i})\bm{\ol{a}}_{n+i}\\
& \quad +\sum_{i=1}^{2k}\ol{\epsilon}_{\tau_i}t^{-\rho(\ol{\alpha}_i)}(s^{\ol{\eta}_i}-t^{\ol{\eta}_i})\bm{\ol{a}}_{2n+i}
+\sum_{i=1}^{2k}\ol{\epsilon}_{\tau_i}t^{-\rho(\ol{\beta}_i)}(s^{\ol{\theta}_i}-t^{\ol{\theta}_i})\bm{\ol{a}}_{2n+2k+i}\\
&= \sum_{i=1}^{n}\ol{\epsilon}_{c_i}t^{-\rho(\ol{w}_i)}(s^{\ol{\phi}_i}-t^{\ol{\phi}_i})\bm{\hat{a}}_i
+\sum_{i=1}^{n}\ol{\epsilon}_{c_i}t^{-\rho(\ol{v}'_i)}(s^{\ol{\psi}_i}-t^{\ol{\psi}_i})\bm{\hat{a}}_{n+i}\\
& \quad +\sum_{i=1}^{2k}\ol{\epsilon}_{\tau_i}t^{-\rho(\ol{\alpha}_i)}(s^{\ol{\eta}_i}-t^{\ol{\eta}_i})\bm{\hat{a}}_{2n+i}
+\sum_{i=1}^{2k}\ol{\epsilon}_{\tau_i}t^{-\rho(\ol{\beta}_i)}(s^{\ol{\theta}_i}-t^{\ol{\theta}_i})\bm{\hat{a}}_{2n+2k+i}
=\bm{0}.
\end{align*}

\noindent
If $\epsilon_{c_1}t^{-\rho(w_1)}(s^{\phi_1}-t^{\phi_1})=0$, 
we have $s^{\phi_1}-t^{\phi_1}=s^a-t^a=0$, 
which implies that $\bm{a}_{n+1}=-\bm{\hat{a}}_1$.
Hence it follows that 
\begin{align*}
|\rank A(D,\phi;X)-\rank A(\ol{D},\ol{\phi};X)|
= |\rank A(D,\phi;X)-\rank \hat{A}(\ol{D},\ol{\phi};X)| \leq 1.
\end{align*}

\noindent
If $\ol{\epsilon}_{c_1}t^{-\rho(\ol{w}_1)}(s^{\ol{\phi}_1}-t^{\ol{\phi}_1})=0$,
we have $s^{\ol{\phi}_1}-t^{\ol{\phi}_1}=s^b-t^b=0$, 
which implies that $\bm{a}_1=-\bm{\hat{a}}_{n+1}$.
Hence it follows that 
\begin{align*}
|\rank A(D,\phi;X)-\rank A(\ol{D},\ol{\phi};X)|
=|\rank A(D,\phi;X)-\rank \hat{A}(\ol{D},\ol{\phi};X)| \leq 1.
\end{align*}

\noindent
If $\epsilon_{c_1}t^{-\rho(w_1)}(s^{\phi_1}-t^{\phi_1}) \neq 0$ and 
$\ol{\epsilon}_{c_1}t^{-\rho(\ol{w}_1)}(s^{\ol{\phi}_1}-t^{\ol{\phi}_1}) \neq 0$,
we can represent $\bm{a}_1$ and $\bm{\ol{a}}_1$ 
as linear combinations of $\bm{a}_2, \ldots, \bm{a}_{2n+4k}$ 
and $\bm{\ol{a}}_2, \ldots, \bm{\ol{a}}_{2n+4k}$ respectively.
Hence it follows that 

\begin{align*}
\rank A(D,\phi;X)=\rank
\begin{pmatrix}
\bm{a}_2\\
\vdots\\
\bm{a}_{2n+4k}
\end{pmatrix},~
\rank A(\ol{D},\ol{\phi};X)=\rank
\begin{pmatrix}
\bm{\ol{a}}_2\\
\vdots\\
\bm{\ol{a}}_{2n+4k}
\end{pmatrix},
\end{align*}
which implies that 
\begin{align*}
|\rank A(D,\phi;X)-\rank A(\ol{D},\ol{\phi};X)|
&=\left |\rank
\begin{pmatrix}
\bm{a}_2\\
\vdots\\
\bm{a}_{2n+4k}
\end{pmatrix}
-\rank
\begin{pmatrix}
\bm{\ol{a}}_2\\
\vdots\\
\bm{\ol{a}}_{2n+4k}
\end{pmatrix} \right |\\
&=\left |\rank
\begin{pmatrix}
\bm{a}_2\\
\vdots\\
\bm{a}_{2n+4k}
\end{pmatrix}
-\rank
\begin{pmatrix}
\bm{\hat{a}}_2\\
\vdots\\
\bm{\hat{a}}_{2n+4k}
\end{pmatrix} \right |\\
&\leq 1.
\end{align*}

Consequently, 
if we can deform $H_1$ into $H_2$ 
by crossing changes at $l$ crossings, 
then for any $\mathbb{Z}_m$-flowed diagram $(D_1,\phi_1)$ of $H_1$, 
there exists a $\mathbb{Z}_m$-flowed diagram $(D_2,\phi_2)$ of $H_2$ 
satisfying $\gcd \phi_1=\gcd \phi_2$ and 
\[ | \dim \col_X(D_1,\phi_1) - \dim \col_X(D_2,\phi_2) | \leq l \]
by Lemma \ref{gcd}.
Therefore 
it follows that 
\begin{align*}
\max_{\phi_1 \in \flow (D_1;\mathbb{Z}_m)} 
\min_{\substack{\phi_2 \in \flow (D_2;\mathbb{Z}_m) \\ \gcd \phi_1=\gcd \phi_2}}
|\dim \col_X(D_1,\phi_1)-\dim \col_X(D_2,\phi_2)| \leq d(H_1,H_2).
\end{align*}
\end{proof}

By Proposition \ref{trivial col.} and Theorem \ref{gordian distance}, 
the following corollary holds immediately.

\begin{corollary}\label{unknotting number}
Let $H$ be an $S^1$-oriented handlebody-knot 
and let $D$ be a diagram of $H$.
Let $X=\mathbb{Z}_p[t^{\pm 1}]/(f(t))$ 
which is a $\mathbb{Z}_m$-family of Alexander biquandles, 
where $p$ is a prime number, $s \in \mathbb{Z}_p[t^{\pm 1}]$ 
and $f(t) \in \mathbb{Z}_p[t^{\pm 1}]$ is an irreducible polynomial.
Then it follows that 
\begin{align*}
\max_{\phi \in \flow (D;\mathbb{Z}_m)} 
\dim \col_X(D,\phi)-1 \leq u(H).
\end{align*}
\end{corollary}

\section{Examples}

In this section, we give some examples.
In Example \ref{ex. Iwakiri}, 
we give a handlebody-knot
with unknotting number $2$, 
and in Remark \ref{rem. Iwakiri}, 
we note that 
it can not be obtained by using Alexander quandle colorings 
with $\mathbb{Z}_2, \mathbb{Z}_3$-flows introduced in \cite{Iwakiri15}.
In Example \ref{ex. unknotting number}, 
we give three handlebody-knots 
with unknotting number $n$ for any $n \in \mathbb{Z}_{>0}$.
In Example \ref{ex. gordian distance}, 
we give two handlebody-knots 
with their Gordian distance $n$ for any $n \in \mathbb{Z}_{>0}$.

\begin{example}\label{ex. Iwakiri}
Let $H$ be the handlebody-knot 
represented by the $\mathbb{Z}_{10}$-flowed diagram $(D,\phi)$
depicted in Figure \ref{H}.
Then we show that $u(H)=2$.

Let $s=1 \in \mathbb{Z}_3[t^{\pm 1}]$ 
and let $f(t)=t^4+2t^3+t^2+2t+1 \in \mathbb{Z}_3[t^{\pm 1}]$, 
which is an irreducible polynomial.
Then $X:=\mathbb{Z}_3[t^{\pm 1}]/(f(t))$ is 
a $\mathbb{Z}_{10}$-family of Alexander biquandles.
Then 
for any $x,y,z \in X$, 
the assignment of them to each semi-arc of $(D,\phi)$ 
as shown in Figure \ref{H} 
is an $X$-coloring of $(D,\phi)$, 
which implies $\dim \col_X(D,\phi) \geq 3$.
By Corollary \ref{unknotting number}, 
we obtain $2 \leq u(H)$.
On the other hand, 
we can deform $H$ into a trivial handlebody-knot 
by the crossing changes 
at two crossings surrounded by dotted circles depicted in Figure \ref{H}.
Therefore it follows that $u(H)=2$.
\end{example}

\begin{figure}[htb]
\begin{center}
\includegraphics[width=70mm]{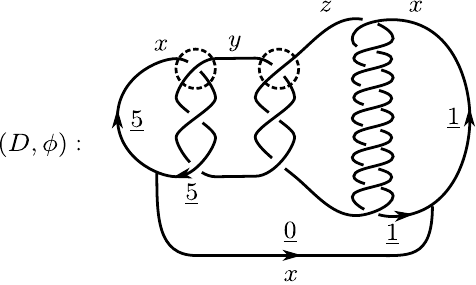}
\end{center}
\caption{A $\mathbb{Z}_{10}$-flowed diagram $(D,\phi)$ of $H$.}\label{H}
\end{figure}

\begin{remark}\label{rem. Iwakiri}
We show that 
the result in Example \ref{ex. Iwakiri} 
can not be obtained by using Alexander quandle colorings 
with $\mathbb{Z}_2, \mathbb{Z}_3$-flows introduced in \cite{Iwakiri15}.

Let $H$ be the handlebody-knot 
represented by the $\mathbb{Z}_m$-flowed diagram $(D,\phi(a,b))$
depicted in Figure \ref{H-2} for any $m=2,3$ and $a,b \in \mathbb{Z}_m$.
Let $p$ be a prime number, 
$s=1 \in \mathbb{Z}_p[t^{\pm 1}]$, 
$f(t)$ be an irreducible polynomial in $\mathbb{Z}_p[t^{\pm 1}]$ 
and let $X=\mathbb{Z}_p[t^{\pm 1}]/(f(t))$ which is 
a $\mathbb{Z}_m$-family of Alexander (bi)quandles.
We note that 
$\col_X(D,\phi(a,b))$ is generated by 
$x,y,z \in X$ 
as shown in Figure \ref{H-2} 
for any $m=2,3$ and $a, b \in \mathbb{Z}_m$.
If $(a,b)=(1,0)$, 
$x$, $y$ and $z$ need to satisfy the following relations: 
\begin{align*}
& (t^2-t+1)x-(t^2-t+1)y=0,\\
& -t(t^2-t+1)x+t^{-1}(t+1)(t-1)(t^2-t+1)y+t^{-1}(t^2-t+1)z=0,\\
& -t^{-1}(t-1)(t^2-t+1)x+t^{-2}(t^2-t-1)(t^2-t+1)y+t^{-2}(t^2-t+1)z=0,\\
& ((t^3+t^2-1)(t^2-t+1)-t)x-((t^3+t^2-1)(t^2-t+1)-t)z=0,
\end{align*}
that is, 
\begin{align*}
M
\begin{pmatrix}
x\\y\\z
\end{pmatrix}
=
\begin{pmatrix}
0\\0\\0\\0
\end{pmatrix},
\end{align*}
where 
\begin{align*}
M =
{\fontsize{9pt}{4mm}\selectfont
\begin{pmatrix}
t^2-t+1 & -(t^2-t+1) & 0\\
-t(t^2-t+1) & t^{-1}(t+1)(t-1)(t^2-t+1) & t^{-1}(t^2-t+1)\\
-t^{-1}(t-1)(t^2-t+1) & t^{-2}(t^2-t-1)(t^2-t+1) & t^{-2}(t^2-t+1)\\
(t^3+t^2-1)(t^2-t+1)-t & 0 & -(t^3+t^2-1)(t^2-t+1)+t
\end{pmatrix}}.
\end{align*}
These relations are obtained from crossings $c_1,c_2,c_3$ and $c_4$ 
as shown in Figure \ref{H-2}.
When $t^2-t+1 \neq 0$ in $X$, 
it is clearly that $\rank M \geq 1$.
When $t^2-t+1 = 0$ in $X$, 
we have 
\begin{align*}
M=
\begin{pmatrix}
0 & 0 & 0\\
0 & 0 & 0\\
0 & 0 & 0\\
-t & 0 & t
\end{pmatrix},
\end{align*}
which implies that 
$\rank M = 1$.
Hence we have 
$\dim \col_X(D,\phi(1,0)) = 3-\rank M \leq 2$.
Therefore we can not obtain $2 \leq u(H)$.

We can prove the remaining cases in the same way.
\end{remark}

\begin{figure}[htb]
\begin{center}
\includegraphics[width=80mm]{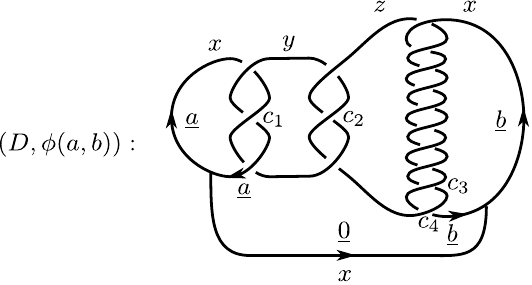}
\end{center}
\caption{A $\mathbb{Z}_m$-flowed diagram $(D,\phi(a,b))$ of $H$.}\label{H-2}
\end{figure}

\begin{example}\label{ex. unknotting number}
Let $A_n$, $B_n$ and $C_n$ be the handlebody-knots 
represented by the $\mathbb{Z}_8$-flowed diagram $(D_{A_n},\phi_{A_n})$, 
the $\mathbb{Z}_{24}$-flowed diagram $(D_{B_n},\phi_{B_n})$ and 
the $\mathbb{Z}_8$-flowed diagram $(D_{C_n},\phi_{C_n})$ 
depicted in Figure \ref{A_n}, \ref{B_n} and \ref{C_n} respectively 
for any $n \in \mathbb{Z}_{>0}$.
Then we show that 
$u(A_n)=u(B_n)=u(C_n)=n$.

\begin{enumerate}
\item
Let $s=t+1 \in \mathbb{Z}_3[t^{\pm 1}]$ 
and let $f(t)=t^2+t+2 \in \mathbb{Z}_3[t^{\pm 1}]$, which is an irreducible polynomial.
Then $X:=\mathbb{Z}_3[t^{\pm 1}]/(f(t))$ is 
a $\mathbb{Z}_8$-family of Alexander biquandles.
Then for any $x_0,x_1,\ldots,x_n \in X$, 
the assignment of them to each semi-arc of $(D_{A_n},\phi_{A_n})$ 
as shown in Figure \ref{A_n} 
is an $X$-coloring of $(D_{A_n},\phi_{A_n})$, 
which implies $\dim \col_X(D_{A_n},\phi_{A_n}) \geq n+1$.
By Corollary \ref{unknotting number}, 
we obtain $n \leq u(A_n)$.
On the other hand, 
we can deform $A_n$ into a trivial handlebody-knot 
by the crossing changes 
at $n$ crossings surrounded by dotted circles depicted in Figure \ref{A_n}.
Therefore it follows that $u(A_n)=n$.

\begin{figure}[htb]
\begin{center}
\includegraphics[width=95mm]{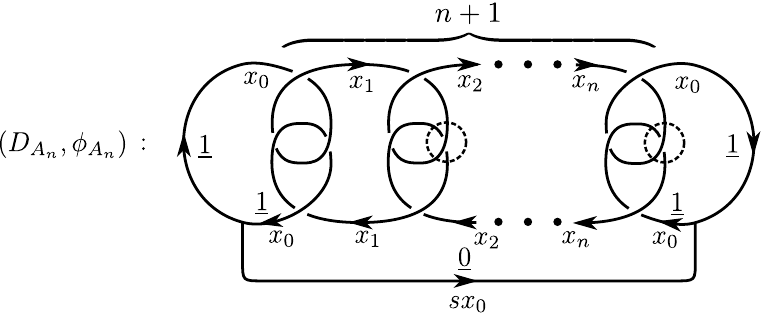}
\end{center}
\caption{A $\mathbb{Z}_8$-flowed diagram $(D_{A_n},\phi_{A_n})$ of $A_n$.}\label{A_n}
\end{figure}

\item
Let $s=t^2+1 \in \mathbb{Z}_5[t^{\pm 1}]$ 
and let $f(t)=t^2+2t+4 \in \mathbb{Z}_5[t^{\pm 1}]$, which is an irreducible polynomial.
Then $X:=\mathbb{Z}_5[t^{\pm 1}]/(f(t))$ is 
a $\mathbb{Z}_{24}$-family of Alexander biquandles.
Then for any $x_0,x_1,\ldots,x_n \in X$, 
the assignment of them to each semi-arc of $(D_{B_n},\phi_{B_n})$ 
as shown in Figure \ref{B_n} 
is an $X$-coloring of $(D_{B_n},\phi_{B_n})$, 
which implies $\dim \col_X(D_{B_n},\phi_{B_n}) \geq n+1$.
By Corollary \ref{unknotting number}, 
we obtain $n \leq u(B_n)$.
On the other hand, 
we can deform $B_n$ into a trivial handlebody-knot 
by the crossing changes 
at $n$ crossings surrounded by dotted circles depicted in Figure \ref{B_n}.
Therefore it follows that $u(B_n)=n$.

\begin{figure}[htb]
\begin{center}
\includegraphics[width=95mm]{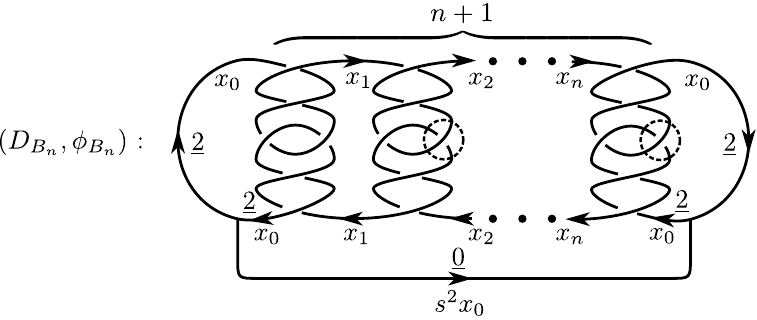}
\end{center}
\caption{A $\mathbb{Z}_{24}$-flowed diagram $(D_{B_n},\phi_{B_n})$ of $B_n$.}\label{B_n}
\end{figure}

\item
Let $s=2t-1 \in \mathbb{Z}_3[t^{\pm 1}]$ 
and let $f(t)=t^2+t+2 \in \mathbb{Z}_3[t^{\pm 1}]$, which is an irreducible polynomial.
Then $X:=\mathbb{Z}_3[t^{\pm 1}]/(f(t))$ is 
a $\mathbb{Z}_8$-family of Alexander biquandles.
Then for any $x_0,x_1,\ldots,x_n \in X$, 
the assignment of them to each semi-arc of $(D_{C_n},\phi_{C_n})$ 
as shown in Figure \ref{C_n} 
is an $X$-coloring of $(D_{C_n},\phi_{C_n})$, 
which implies $\dim \col_X(D_{C_n},\phi_{C_n}) \geq n+1$.
By Corollary \ref{unknotting number}, 
we obtain $n \leq u(C_n)$.
On the other hand, 
we can deform $C_n$ into a trivial handlebody-knot 
by the crossing changes 
at $n$ crossings surrounded by dotted circles depicted in Figure \ref{C_n}.
Therefore it follows that $u(C_n)=n$.

\begin{figure}[htb]
\begin{center}
\includegraphics[width=95mm]{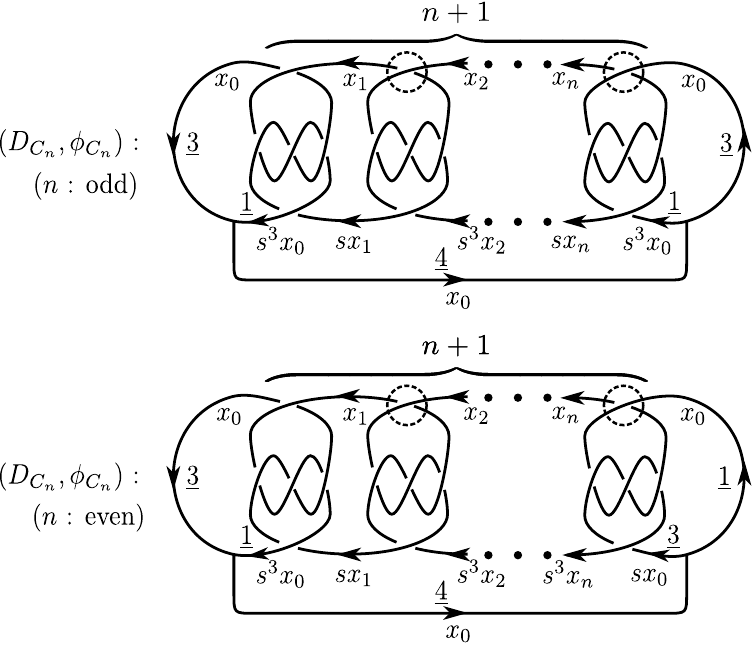}
\end{center}
\caption{A $\mathbb{Z}_8$-flowed diagram $(D_{C_n},\phi_{C_n})$ of $C_n$.}\label{C_n}
\end{figure}
\end{enumerate}
\end{example}

\begin{example}\label{ex. gordian distance}
Let $H_n$ and $H'_n$ be the handlebody-knots 
represented by the $\mathbb{Z}_3$-flowed diagrams 
$(D_n,\phi_n)$ and $(D'_n,\phi'_n(a,b))$ respectively 
depicted in Figure \ref{H_n,H'_n} 
for any $n \in \mathbb{Z}_{>0}$ and $a,b \in \mathbb{Z}_3$.
Then we show that 
$d(H_n,H'_n)=n$.

Let $s=1 \in \mathbb{Z}_2[t^{\pm 1}]$ 
and let $f(t)=t^2+t+1 \in \mathbb{Z}_2[t^{\pm 1}]$, which is an irreducible polynomial.
Then $X:=\mathbb{Z}_2[t^{\pm 1}]/(f(t))$ is 
a $\mathbb{Z}_3$-family of Alexander (bi)quandles.
Then for any $x_0,x_1,\ldots,x_n, y_1, \ldots, y_n \in X$, 
the assignment of them to each semi-arc of $(D_n,\phi_n)$ 
as shown in Figure \ref{H_n,H'_n} 
is an $X$-coloring of $(D_n,\phi_n)$, 
which implies $\dim \col_X(D_n,\phi_n) \geq 2n+1$.

On the other hand, 
we note that  
$\col_X(D'_n,\phi'_n(a,b))$ is generated by 
$x_0,x_1,x'_1,\ldots, x_n,$ $x'_n, y_1,y'_1,\ldots, y_n,y'_n \in X$ 
as shown in Figure \ref{H_n,H'_n} 
for any $a, b \in \mathbb{Z}_3$.
If $(a,b)=(0,0)$, 
it is easy to see that $\dim \col_X(D'_n,\phi'_n(a,b))=1$.
If $(a,b)=(1,1), (1,2), (2,1), (2,2)$, 
we obtain that $x_i=x'_i=y_i=y'_i$ for any $i=1,2,\ldots, n$, 
which implies $\dim \col_X(D'_n,\phi'_n(a,b)) \leq n+1$.
If $(a,b)=(0,1), (0,2)$,
we have 
\begin{align*}
& x_0=x_1=x_2,\\
& x_{i+2}=x'_i ~(i=1,2,\ldots, n-2),\\
& x'_i=
\begin{cases}
x_i \us^b y'_i ~(i:\mathrm{odd}),\\
x_i \us^{-b} y'_i ~(i:\mathrm{even}),
\end{cases}\\
& x_n=x'_{n-1},\\
& y_i=y'_i ~(i=1,2,\ldots, n).
\end{align*}
Hence $\col_X(D'_n,\phi'_n(a,b))$ is generated by $x_0,y_1,\ldots, y_n \in X$, 
which implies $\dim \col_X(D'_n,\phi'_n(a,b)) \leq n+1$.
If $(a,b)=(1,0), (2,0)$, 
in the same way as when $(a,b)=(0,1), (0,2)$, 
$\col_X(D'_n,\phi'_n(a,b))$ is generated by $x_0,x_1,\ldots x_n \in X$, 
which implies $\dim \col_X(D'_n,\phi'_n(a,b)) \leq n+1$.
Hence for any $a,b \in \mathbb{Z}_3$, 
$\dim \col_X(D'_n,\phi'_n(a,b)) \leq n+1$, 
which implies that 
\begin{align*}
\dim \col_X(D_n,\phi_n)-\dim \col_X(D'_n,\phi'_n(a,b)) \geq n.
\end{align*}
By Theorem \ref{gordian distance}, 
it follows that $n \leq d(H_n,H'_n)$.

Finally, 
we can deform $H'_n$ into $H_n$ by the crossing changes 
at $n$ crossings surrounded by dotted circles depicted in Figure \ref{H_n,H'_n}.
Therefore it follows that $d(H_n,H'_n)=n$.

\end{example}

\begin{figure}[htb]
\begin{center}
\includegraphics[width=100mm]{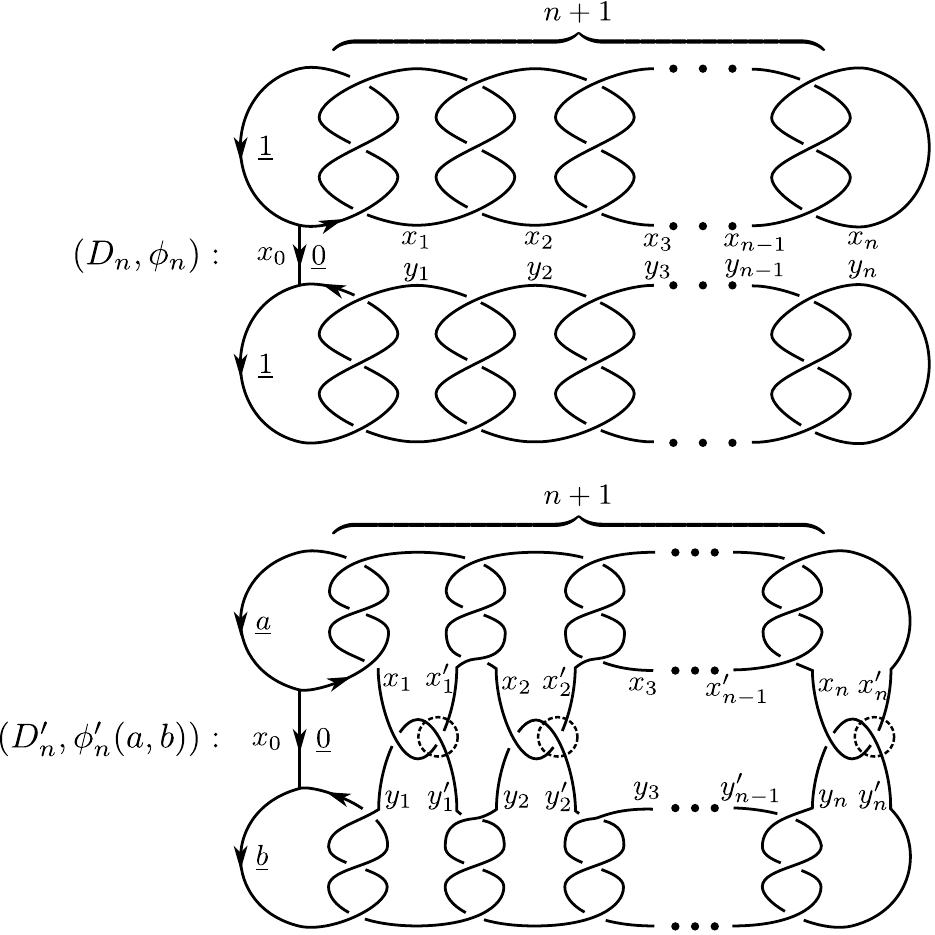}
\end{center}
\caption{$\mathbb{Z}_3$-flowed diagrams $(D_n,\phi_n)$ and $(D'_n,\phi'_n)$ of $H_n$ and $H'_n$.}\label{H_n,H'_n}
\end{figure}

\section*{Acknowledgment}
The author would like to thank Masahide Iwakiri for his helpful comments.
He is particularly grateful to Atsushi Ishii for invaluable advice and his suggestions.


\end{document}